\documentclass[oneside,english]{amsart}
\usepackage[T1]{fontenc}
\usepackage[latin9]{inputenc}
\usepackage{amsthm}
\usepackage{amssymb}

\makeatletter
\numberwithin{equation}{section}
\numberwithin{figure}{section}
\theoremstyle{plain}
\newtheorem{thm}{\protect\theoremname}
  \theoremstyle{definition}
  \newtheorem{defn}[thm]{\protect\definitionname}
  \theoremstyle{remark}
  \newtheorem*{rem*}{\protect\remarkname}
  \theoremstyle{remark}
  \newtheorem{rem}[thm]{\protect\remarkname}
  \theoremstyle{plain}
  \newtheorem{prop}[thm]{\protect\propositionname}
  \theoremstyle{remark}
  \newtheorem{claim}[thm]{\protect\claimname}
  \theoremstyle{plain}
  \newtheorem{lem}[thm]{\protect\lemmaname}
  \theoremstyle{plain}
  \newtheorem{cor}[thm]{\protect\corollaryname}

\@ifundefined{definecolor}{\usepackage{color}}{}
\@ifundefined{definecolor}{\@ifundefined{definecolor}
 {\usepackage{color}}{}
}{}

\usepackage{babel}

  \providecommand{\claimname}{Claim}
  \providecommand{\definitionname}{Definition}
  \providecommand{\remarkname}{Remark}
\providecommand{\theoremname}{Theorem}

\usepackage{babel}

  \providecommand{\claimname}{Claim}
  \providecommand{\corollaryname}{Corollary}
  \providecommand{\definitionname}{Definition}
  \providecommand{\lemmaname}{Lemma}
  \providecommand{\propositionname}{Proposition}
  \providecommand{\remarkname}{Remark}
\providecommand{\theoremname}{Theorem}

\usepackage{babel}\providecommand{\claimname}{Claim}
  \providecommand{\corollaryname}{Corollary}
  \providecommand{\definitionname}{Definition}
  \providecommand{\lemmaname}{Lemma}
  \providecommand{\propositionname}{Proposition}
  \providecommand{\remarkname}{Remark}
\providecommand{\theoremname}{Theorem}

\makeatother

\usepackage{babel}
  \providecommand{\claimname}{Claim}
  \providecommand{\corollaryname}{Corollary}
  \providecommand{\definitionname}{Definition}
  \providecommand{\lemmaname}{Lemma}
  \providecommand{\propositionname}{Proposition}
  \providecommand{\remarkname}{Remark}
\providecommand{\theoremname}{Theorem}

\begin{document}

\title{Calder\'on couples of $p$-convexified Banach lattices}

\author{Eliran Avni and Michael Cwikel}

\address{Department of Mathematics, Technion - Israel Institute of Technology,
Haifa 32000, Israel}

\email{eliran5@tx.technion.ac.il, mcwikel@math.technion.ac.il}



\thanks{The research of the first named author was supported by a Chester and Taube Hurwitz Foundation Fellowship. The research of the second named author
was supported by funding from the Martin and Sima Jelin Chair in Mathematics
at the Technion. }

\begin{abstract}
{\normalsize We deal with the question of whether the $p$-convexified
couple $\left(X_{0}^{(p)},X_{1}^{(p)}\right)$ is a Calder\'on couple
under the assumption that $\left(X_{0},X_{1}\right)$ is a Calder\'on
couple of Banach lattices on some measure space. We find that the
answer is affirmative whenever the spaces $X_{0},\, X_{1}$ are complete
lattices and an additional ``positivity'' assumption is imposed
regarding $\left(X_{0},X_{1}\right)$. We also prove a quantitative
version of the result with appropriate norm estimates. In the appendix
we identify some cases where appropriate assumptions on a Banach lattice
$X$ guarantee that it is indeed a complete lattice.}{\normalsize \par}
\end{abstract}
\maketitle

\section{preliminaries, definitions, notations, conventions}
\begin{defn}
A \textbf{Banach lattice of measurable functions} $X$ is a Banach
space of (equivalence classes of) measurable functions defined on
a certain measure space $(\Omega,\Sigma,\mu)$ and taking values in
$\mathbb{R}$ or $\mathbb{C}$ (in this paper, in $\mathbb{R}$),
with the following property: if $f,g:\Omega\rightarrow\mathbb{R}$
are two measurable functions, and if $f\in X$ and $\vert g\vert\leq\vert f\vert$
almost everywhere then we also have $g\in X$ and $\parallel g\parallel\leq\parallel f\parallel$.
In this paper we will usually use the shorter terminology ``Banach
lattice'' although in other settings this is used in a more abstract
context (see e.g. \cite{LT} Definition 1.a.1 p.~1). 
\end{defn}
\begin{defn} For each Banach lattice $X$ of measurable functions
on a measure space $\left(\Omega,\Sigma,\mu\right)$ and each $p\in(1,\infty)$
we recall that the \textbf{$p$-convexification of $X$} is the set
$X^{(p)}$ of all measurable functions $f:\Omega\to\mathbb{R}$ for
which $\left|f\right|^{p}\in X$. When endowed with the norm 
$\left\Vert f\right\Vert _{X^{(p)}}
=\left(\left\Vert \left|f\right|^{p}\right\Vert ^{1/p}\right)$
it is also a Banach lattice. \end{defn}

\begin{defn} Whenever $X_{0},\, X_{1}$ are two Banach lattices with
the same underlying measure space $(\Omega,\Sigma,\mu)$ we define
$X_{0}+X_{1}$ to be the space of all measurable functions $f:\Omega\rightarrow\mathbb{R}$
for which there are $a_{j}\in X_{j}$ ($j=0,1$) such that $f=a_{0}+a_{1}$.
This is a Banach space (in fact a Banach lattice), when endowed with
the following norm: 
\begin{equation}
\left\Vert f\right\Vert _{X_{0}+X_{1}}=\inf\left\{ \left\Vert a_{0}\right\Vert _{X_{0}}+\left\Vert a_{1}\right\Vert _{X_{1}}\vert a_{j}\in X_{j}\,,\, j=0,1\,,\, f=a_{0}+a_{1}\right\} \label{definition of the X_0+X_1 norm}
\end{equation}

\end{defn}
\begin{rem*}
Proofs that (\ref{definition of the X_0+X_1 norm}) is a norm rather
than merely a seminorm can be found, e.g. in \cite{cn} Remark 1.41
pp.~34-35 or \cite{kps} Corollary 1, p.~42. This fact implies that
$\left(X_{0},X_{1}\right)$ is a \textbf{Banach couple}, i.e. that
there exists some topological Hausdorff vector space $\mathcal{X}$
such that $X_{0}$ and $X_{1}$ are both continuously embedded in
$\mathcal{X}$ (clearly one can choose $\mathcal{X}=X_{0}+X_{1}$).

In a more general context, whenever $\left(X_{0},X_{1}\right)$ is
a Banach couple, the aforementioned space $X_{0}+X_{1}$ is a Banach
space in which $X_{0}$ and $X_{1}$ are continuously embedded (see
e.g. \cite{bl,BrudnyiKruglyak}). 
\end{rem*}
\begin{defn} For each fixed $t>0$ the following functional 
\[
K(t,f;X_{0},X_{1})=\mbox{inf}\left\{ \left\Vert a_{0}\right\Vert _{X_{0}}+t\left\Vert a_{1}\right\Vert _{X_{1}}\vert 
\, \,
a_{j}\in X_{j},j=0,1,f=a_{0}+a_{1}\right\} 
\]
 is equivalent to the norm (\ref{definition of the X_0+X_1 norm})
and is known as the \textbf{Peetre $K$-functional} (see e.g. \cite{bl,BrudnyiKruglyak}).
\end{defn}

\begin{defn} \label{defn:T:(X_0,X_1) into itself is bounded and linear}The
statement ``\textbf{$T:(X_{0},X_{1})\rightarrow(X_{0},X_{1})$ is
a bounded linear operator}'' means that $T$ is a linear operator
from $X_{0}+X_{1}$ into itself such that the restriction of $T$
to $X_{j}$ is a bounded operator from $X_{j}$ into itself (for $j=0,1$).
\end{defn}

\begin{rem*} We remark that if $T:(X_{0},X_{1})\rightarrow(X_{0},X_{1})$
is a bounded linear operator then automatically $T$ is also a bounded
linear operator from $X_{0}+X_{1}$ into itself, and the following
inequality holds: 
\[
\Vert T\Vert_{X_{0}+X_{1}\to X_{0}+X_{1}}\leq\max\left\{ \Vert T\Vert_{X_{0}\rightarrow X_{0}},\Vert T\Vert_{X_{1}\rightarrow X_{1}}\right\} \,.
\]

\end{rem*}

We recall that $X_{0}\cap X_{1}$ , when endowed with the norm 
$$\left\Vert x\right\Vert _{X_{0}\cap X_{1}}
=\mbox{max}\left\{ \left\Vert x\right\Vert _{X_{0}},\left\Vert x
\right\Vert _{X_{1}}\right\},$$
is also a Banach space. This will be relevant in the following definition.

\begin{defn} Whenever $X_{0}$ and $X_{1}$ are two Banach spaces
continuously embedded in some topological Hausdorff vector space $\mathcal{X}$,
the statement ``\textbf{$A$ is an interpolation space with respect
to $(X_{0},X_{1})$}'' is a concise way to say the following: $A$
is a Banach space satisfying $X_{0}\cap X_{1}\subseteq A\subseteq X_{0}+X_{1}$
where all the inclusions are continuous, and the restriction to $A$
of every bounded linear operator $T:(X_{0},X_{1})\rightarrow(X_{0},X_{1})$
is a bounded operator from $A$ into itself. \end{defn}

\begin{rem*} A Banach space $A$ satisfying $X_{0}\cap X_{1}\subseteq A\subseteq X_{0}+X_{1}$
where all the inclusions are continuous is also called an \textbf{intermediate
space} with respect to $(X_{0},X_{1})$.\end{rem*}

\begin{defn} \label{rdcc}The statement ``\textbf{$(X_{0},X_{1})$
is a Calder\'on couple}'' means that the Banach couple $(X_{0},X_{1})$
has the following property: if $f,g\in X_{0}+X_{1}$ and if $K(t,g;X_{0},X_{1})\leq K(t,f;X_{0},X_{1})$
for every $t>0$, then there exists a bounded linear operator $T:(X_{0},X_{1})\rightarrow(X_{0},X_{1})$
such that $Tf=g$.

Let $C$ be a positive constant. Then the statement ``\textbf{$(X_{0},X_{1})$
is a $C$-Calder\'on couple}'' means that $\left(X_{0},X_{1}\right)$
has the above property, and furthermore the operator $T$ with the
above properties can also be assumed to satisfy $\left\Vert T\right\Vert _{X_{j}\to X_{j}}\le C$
for $j=0,1$. \end{defn}

\begin{rem*} In a number of papers, various alternative terminologies
are used for the notion of a Calder\'on couple. These include $C$-couple,
$K$-adequate couple, $K$-monotone couple, Calder\'on-Mityagin couple
and $\mathcal{CM}$ couple. Of course the interesting and well known
property of Calder\'on couples is that all their interpolation spaces
can be characterized by a simple monotonicity property in terms of
the $K$-functional 
(see e.g. \cite{bl} or \cite{BrudnyiKruglyak}
or \cite{cn} and many of the references therein). 

In fact, it can easily be shown that $\left(X_{0},X_{1}\right)$ is
a Calder\'on couple if and only if all the interpolation spaces $X$
of $\left(X_{0},X_{1}\right)$ are precisely those intermediate spaces
which satisfy the following condition: For every $f,g\in X_{0}+X_{1}$,
if $f\in X$ and if $K(t,g;X_{0},X_{1})\leq K(t,f;X_{0},X_{1})$ for
every $t>0$ then $g\in X$ as well. 

Indeed the property that all the interpolation spaces of $\left(X_{0},X_{1}\right)$
are characterized by the above-mentioned condition is usually taken
to be the \textbf{definition} of a Calder\'on couple. We have simply
found it more convenient here to use an alternative but clearly equivalent
definition.\end{rem*}

Most of the definitions in this section appear extensively in the
literature, but the following one is perhaps new. It relates to a
notion which has been considered in a so far unpublished paper \cite{cwikelmastylo}.

\begin{defn} \label{dcc}The statement ``\textbf{$(X_{0},X_{1})$
is a positive Calder\'on couple}'' means that $(X_{0},X_{1})$ is a
Banach couple of Banach lattices on the same underlying measure space
with the following property: If $f,g\in X_{0}+X_{1}$ and if $K(t,g;X_{0},X_{1})\leq K(t,f;X_{0},X_{1})$
for every $t>0$ and if also $f,g\geq0$ then there exists a \textbf{positive}
bounded linear operator $T:(X_{0},X_{1})\rightarrow(X_{0},X_{1})$
such that $Tf=g$. ($T$ is positive in the sense that if $h\geq0$
a.e. then $Th\geq0$ a.e.)

Analogously to before, the statement ``\textbf{$(X_{0},X_{1})$ is
a positive $C$-Calder\'on couple}'' means that $\left(X_{0},X_{1}\right)$
is a positive Calder\'on couple and, furthermore, the operator $T$
with the above properties can also be assumed to satisfy $\left\Vert T\right\Vert _{X_{j}\to X_{j}}\le C$
for $j=0,1$. \end{defn}

\begin{rem*} Using the fact that pointwise multiplication by a unimodular
measurable function is a norm one linear operator on any Banach lattice,
it is clear that if $\left(X_{0},X_{1}\right)$ is a positive Calder\'on
couple then it is also a Calder\'on couple in the usual sense. Similarly
a positive $C$-Calder\'on couple is a $C$-Calder\'on couple.

The reverse implications are not true. Although all 
known
``natural'' examples of Calder\'on couples of Banach lattices are also positive
Calder\'on couples, it is possible to produce an example of a $C$-Calder\'on
couple of lattices which is not a positive Calder\'on couple. It can
be constructed via a result of Lozanovskii, using a slight variant
of an example in the last section of \cite{cwikelmastylo}. \end{rem*}

\begin{defn} The statement ``\textbf{A Banach lattice $X$ has the
Least Upper Bound Property (or LUBP)}'' means that every subset of
$X$ which is bounded from above has a least upper bound. More precisely,
if $Q$ is a subset of $X$, and if there exists an element $x\in X$
such that $q\le x$ for any $q\in Q$ then there is an element $y\in X$
such that

\[
\begin{array}{cl}
\mbox{(i)} & q\le y\mbox{ for every }q\in Q\:\mbox{ and }\\
\\
\mbox{(ii)} & \mbox{If an element }z\in X\mbox{ satisfies }q\le z\mbox{ for every }q\in Q\mbox{ then }y\le z\,.
\end{array}
\]

A Banach lattice which has the LUBP is also called a ``\textbf{complete
lattice}''
or a ``\textbf{Dedekind complete lattice}".
\end{defn}
\begin{rem}
\label{rmk: examples of complete lattices. }There are many well known
examples of complete lattices. For instance, the lattice $L_{p}(\Omega,\Sigma,\mu)$
is complete whenever $1<p<\infty$, and the lattice $L_{\infty}(\Omega,\Sigma,\mu)$
is complete whenever the measure space $(\Omega,\Sigma,\mu)$ is $\sigma$-
finite (see \cite{Dunford and Schwartz}, Chapter 4, section 8, Theorems
22, 23, p.~302). In addition, a Banach lattice of measurable functions
is a complete lattice whenever it is separable, or whenever the underlying
measure space is $\sigma$-finite (see appendix).
\end{rem}
\begin{defn} (Cf.~\cite{Ioffe}) The statement ``\textbf{a Banach
lattice $X$ has the Hahn-Banach Extension Property (or HBEP)}''
means the following: For any linear space $Z$, any subspace $Y\subseteq Z$,
and any sublinear operator $p:Z\rightarrow X$, if $f:Y\rightarrow X$
is a linear operator such that $\left|f(y)\right|\leq p(y)$ for every
$y\in Y$, then there is a linear operator $F:Z\rightarrow X$ such
that $\left|F(x)\right|\leq p(x)$ for every $x\in Z$ and $f(y)=F(y)$
for every $y\in Y$. \end{defn}

We would also like to mention the following two results, that we shall
resort to later on:
\begin{prop}
\label{mikowski}Assume $X$ is a Banach lattice defined on a measure
space $(\Omega,\Sigma,\mu)$ and $G:X\rightarrow X$ is a positive
linear operator. Then, for every $1<p<\infty$ and every two measurable
functions $h_{1},h_{2}:\Omega\rightarrow\mathbb{R}$ such that $\vert h_{1}\vert^{p},\vert h_{2}\vert^{p}\in X$
we have the pointwise almost everywhere inequality

\[
\left(G(\vert h_{1}+h_{2}\vert^{p})\right)^{\frac{1}{p}}\leq\left(G(\vert h_{1}\vert^{p})\right)^{\frac{1}{p}}+\left(G(\vert h_{2}\vert^{p})\right)^{\frac{1}{p}}
\]

\end{prop}
The proof of this proposition appears in many publications, the earliest
of which we are aware is \cite{Bourbaki}, Chapter 1, Section 2, Proposition
3 dating from the 1950's (However in the bibiography we list a new
English translation of this book).

\begin{thm} \label{thm:LUBP implies HBEP}Every complete Banach lattice
has the HBEP. \end{thm}

The proof of this theorem apparently dates back to L.~V.~Kantorovic'
famous paper from 1935 (see \cite{Kantorovich}). We refer to \cite{Ioffe}
for a short discussion of the history of 
this result.

\begin{rem*} In fact, a Banach lattice has the HBEP if and only if
it is a complete lattice. A proof of this theorem can be found in
a series of papers by W.~Bonnice, R.~Silverman, T.~O.~To and T.~Yen
(see \cite{Silverman and Yen,Bonnice and Silverman1,Bonnice and silverman2}
and \cite{To}). Another source one may refer to is \cite{M.M. Day},
Chapter 4, Section 3, pp.~135-137. A.~D.~Ioffe used a different
approach in proving the same theorem. His proof can be found in \cite{Ioffe}.\end{rem*}

\section{\label{sec:main}THE MAIN PART }

In this section we prove Theorem \ref{thm:a banach couple remains a banach couple},
which is the main result of this paper. It is clear that an analogous
result can be readily obtained in the context of relative Calder\'on
couples, i.e.~where the relevant operators map between two possibly
different couples $\left(X_{0},X_{1}\right)$ and $\left(Y_{0},Y_{1}\right)$.
The definition of relative Calder\'on couples or an equivalent variant
of it, sometimes with different terminology, and often with additional
results about these couples, can be found in many papers, e.g., \cite{BrudnyiKruglyak}
(Definition 4.4.3 p.~579) or \cite{bl} pp.~83-84 or 
\cite{CwMonPros2}
pp.~123-124 or \cite{cn} p.~29 or \cite{cp}
Section 4, pp.~28-39.
For simplicity of presentation we only consider the case where $\left(X_{0},X_{1}\right)=\left(Y_{0},Y_{1}\right)$.

\begin{thm} \label{thm:a banach couple remains a banach couple}Suppose
$(X_{0},X_{1})$ is a positive Calder\'on couple of Banach lattices
defined on an underlying measure space $(\Omega,\Sigma,\mu)$. Suppose
$X_{0}$ and $X_{1}$ are complete lattices. Then, $(X_{0}^{(p)},X_{1}^{(p)})$
is a Calder\'on couple for each $p\in(1,\infty)$.

If, furthermore, $\left(X_{0},X_{1}\right)$ is a positive $C$-Calder\'on
couple then $(X_{0}^{(p)},X_{1}^{(p)})$ is a $2^{1-\frac{1}{p}}C^{\frac{1}{p}}$-Calder\'on
couple. \end{thm}

Before proving the theorem, a few remarks:

For every $f\in X_{0}+X_{1}$ we define the following counterpart
of the $K$-functional:
\begin{eqnarray*}
&  & D(t,f;X_{0},X_{1})\\
& = & \mbox{inf}\left\{ \left\Vert a_{0}\right\Vert _{X_{0}}+
t\left\Vert a_{1}\right\Vert _{X_{1}}
\vert 
\, \,
a_{j}\in X_{j},\, j=0,1,\, f=a_{0}+a_{1},\, a_{0}\cdot a_{1}=0\right\} \,.
\end{eqnarray*}

It is well known that for every $t>0$ and $f\in X_{0}+X_{1}$ the inequality
\begin{equation}
K(t,f;X_{0},X_{1})\leq D(t,f;X_{0},X_{1})\leq2K(t,f;X_{0},X_{1})\label{equi. of K and D}
\end{equation}
 holds.

The straightforward proof of (\ref{equi. of K and D}) appears essentially
as part of the proof of Lemma 4.3 on p.~310 of \cite{N1} and is
also given on pp.~280-281 of \cite{CerdaCollMartin}. (The additional
assumptions made in the context of Lemma 4.3 of \cite{N1} do not
effect the validity of the argument in a more general setting.) 

\begin{claim} \label{claim: equi. of D(t,|f|^p,X)^(1/p) and D(t^(1/p),f,X^p)}For
a measurable function $f:\Omega\rightarrow\mathbb{R}$, $f\in X_{0}^{(p)}+X_{1}^{(p)}$
iff $\left|f\right|^{p}\in X_{0}+X_{1}$. In addition, for every $1<p<\infty$
the following inequality is valid 
\[
\left(D(t,\left|f\right|^{p};X_{0},X_{1})\right)^{\frac{1}{p}}\leq D(t^{\frac{1}{p}},f;X_{0}^{(p)},X_{1}^{(p)})\leq2^{1-\frac{1}{p}}\left(D(t,\left|f\right|^{p};X_{0},X_{1})\right)^{\frac{1}{p}}\,.
\]

\end{claim}

\begin{rem*} It has been mentioned without proof in \cite[p. 289]{CerdaCollMartin}
that the functionals $K(t,\vert f\vert^{p};X_{0},X_{1})$ and $\left(K(t^{\frac{1}{p}},f;X_{0}^{(p)},X_{1}^{(p)})\right)^{p}$
are equivalent. In fact, combining Claim \ref{claim: equi. of D(t,|f|^p,X)^(1/p) and D(t^(1/p),f,X^p)}
with (\ref{equi. of K and D}) immediately gives us 
\begin{equation}
K(t,\vert f\vert^{p};X_{0},X_{1})\leq2^{p}\left(K(t^{\frac{1}{p}},f;X_{0}^{(p)},X_{1}^{(p)})\right)^{p}\leq2^{2p}K(t,\vert f\vert^{p};X_{0},X_{1})\,.\label{eq:equi. of K(t,|f|^p,X) and (K,t^(1/p),f,X^p))^p}
\end{equation}

\end{rem*}

The proof of Claim \ref{claim: equi. of D(t,|f|^p,X)^(1/p) and D(t^(1/p),f,X^p)}
and thus of the equivalence (\ref{eq:equi. of K(t,|f|^p,X) and (K,t^(1/p),f,X^p))^p})
is an easy exercise and is left to the reader (see also \cite{Avni}).
In fact, L.~Maligranda has proved the following stronger version
of inequality (\ref{eq:equi. of K(t,|f|^p,X) and (K,t^(1/p),f,X^p))^p}).

\begin{equation}
\left(K(t,\vert f\vert^{p};X_{0},X_{1})\right)^{\frac{1}{p}}\leq K(t^{\frac{1}{p}},f;X_{0}^{(p)},X_{1}^{(p)})\leq2^{1-\frac{1}{p}}\left(K(t,\vert f\vert^{p};X_{0},X_{1})\right)^{\frac{1}{p}}\,,\label{eq:Maligranda}
\end{equation}
 and kindly shown us the proof in a private communication. His result
is mentioned without proof in \cite{Maligranda}.

The following two claims prove that $X_{0}^{(p)}+X_{1}^{(p)}$ is
a complete lattice if and only if $X_{0}$ and $X_{1}$ are complete
lattices. \begin{claim} \label{Clm: X is complete iff X^p is-1}If
$X$ is a Banach lattice of measurable functions then $X$ is a complete
lattice if and only if $X^{(p)}$ is a complete lattice. \end{claim}
We postpone the easy proof of this claim to the appendix (see Remark
\ref{Proof that X is a complete lattice iff X^(p) is}). \begin{claim}
\label{Clm: X_0+X_1 is complete iff X_0, X_1 are-1}Assume that $X_{0}$
and $X_{1}$ are two Banach lattices defined on the same underlying
measure space. Then $X_{0}+X_{1}$ is a complete lattice if and only
if $X_{0}$ and $X_{1}$ are complete lattices. \end{claim} Here
again we refer the reader to the appendix for a proof of this claim.

We now turn to the proof of Theorem \ref{thm:a banach couple remains a banach couple}.

\noindent \textit{Proof.} We start by assuming that $f,g\in X_{0}^{(p)}+X_{1}^{(p)}$
and that $K(t,g;X_{0}^{(p)},X_{1}^{(p)})\leq K(t,f;X_{0}^{(p)},X_{1}^{(p)})$
for every $t>0$. We need to prove that there exists a linear operator
$L:(X_{0}^{(p)},X_{1}^{(p)})\rightarrow(X_{0}^{(p)},X_{1}^{(p)})$
such that $Lf=g$.

It follows from (\ref{eq:equi. of K(t,|f|^p,X) and (K,t^(1/p),f,X^p))^p})
or from (\ref{eq:Maligranda}) that there is a constant $\alpha>0$
such that 
\[
K(t,\vert g\vert^{p};X_{0},X_{1})\leq K(t,\alpha\vert f\vert^{p};X_{0},X_{1})
\]
 for all $t>0$.

According to our assumption, since $(X_{0},X_{1})$ is a positive
Calder\'on couple, there exists a bounded linear positive operator $T:(X_{0},X_{1})\rightarrow(X_{0},X_{1})$
such that $T(\alpha\vert f\vert^{p})=\vert g\vert^{p}$. If, furthermore,
$\left(X_{0},X_{1}\right)$ is a positive $C$-Calder\'on couple then
we can also assert that 
\begin{equation}
\left\Vert T\right\Vert _{X_{j}\to X_{j}}\le C\mbox{ for }j=0,1\,.\label{eq:temp}
\end{equation}

Let us now define $H:X_{0}^{(p)}+X_{1}^{(p)}\rightarrow X_{0}^{(p)}+X_{1}^{(p)}$
by setting 
\[
H(h)=\left(T(\alpha\vert h\vert^{p})\right)^{\frac{1}{p}}
\]
 for every $h\in X_{0}^{(p)}+X_{1}^{(p)}$ (Since $T$ is positive
and $\vert h\vert^{p}\geq0$, the expression $\left(T(\alpha\vert h\vert^{p})\right)^{\frac{1}{p}}$
is meaningful).

According to Claim \ref{claim: equi. of D(t,|f|^p,X)^(1/p) and D(t^(1/p),f,X^p)},
it is obvious that $H(h)\in X_{0}^{(p)}+X_{1}^{(p)}$.

Then we observe that 
\begin{eqnarray*}
H(f) & = & \left(T(\alpha\vert f\vert^{p})\right)^{\frac{1}{p}}\\
 & = & \vert g\vert^{p\cdot\frac{1}{p}}\\
 & = & \vert g\vert\,.
\end{eqnarray*}
 It is easy to check that $H$ is sublinear, that is: 
\begin{itemize}
\item For every $\lambda\in\mathbb{R}$ we have 
\begin{equation}
H(\lambda h)=\vert\lambda\vert H(h)\,.\label{eq:H(x_h)equals0003Dx_H(h)}
\end{equation}

\item For every $h_{1},h_{2}\in X_{0}^{(p)}+X_{1}^{(p)}$ we have 
\begin{equation}
H(h_{1}+h_{2})\leq H(h_{1})+H(h_{2})\,.\label{eq:H(h_1+h_2) Lequals0003D H(h_1)+H(h_2)}
\end{equation}

\end{itemize}
We will need the sublinearity of $H$ in order to apply Theorem \ref{thm:LUBP implies HBEP}
in a moment.

\smallskip{}

(\ref{eq:H(x_h)equals0003Dx_H(h)}) is immediate and (\ref{eq:H(h_1+h_2) Lequals0003D H(h_1)+H(h_2)})
follows from Proposition \ref{mikowski} and the fact that $T$ is
positive and linear.

We now define $l:\mbox{Span}\{f\}\rightarrow X_{0}^{(p)}+X_{1}^{(p)}$
by setting 
\[
l(\lambda f)=\lambda g
\]
for all $\lambda\in\mathbb{R}$.

We obviously have

\[
\vert l(\lambda f)\vert=\vert\lambda g\vert=\vert\lambda\vert H(f)=H(\lambda f)
\]
for every $\lambda\in\mathbb{R}$.

Since we assume $X_{0}$ and $X_{1}$ are complete lattices, Claim
\ref{Clm: X is complete iff X^p is-1} and Claim \ref{Clm: X_0+X_1 is complete iff X_0, X_1 are-1}
imply that $X_{0}^{(p)}+X_{1}^{(p)}$ is a complete lattice too. Since
$H$ is sublinear, Theorem \ref{thm:LUBP implies HBEP} guarantees
the existence of a linear operator $L:X_{0}^{(p)}+X_{1}^{(p)}\rightarrow X_{0}^{(p)}+X_{1}^{(p)}$
that extends $l$ and for which 
\[
\vert L(h)\vert\leq H(h)
\]
for every $h\in X_{0}^{(p)}+X_{1}^{(p)}$.

Note that $L(f)=g$.

Finally, to complete the proof of Theorem \ref{thm:a banach couple remains a banach couple},
we will show that the restriction of $L$ to $X_{j}^{(p)}$ (for $j=0,1$)
is a bounded linear operator into $X_{j}^{(p)}$ and estimate its
norm.

Let us therefore assume $h\in X_{j}^{(p)}$. We may write 
\begin{eqnarray*}
\left|L(h)\right| & \leq & H(h)\\
 & = & \left(T(\alpha\vert h\vert^{p})\right)^{\frac{1}{p}}\,.
\end{eqnarray*}

Since $h\in X_{j}^{(p)}$, it is also true that $\vert h\vert^{p}\in X_{j}$,
and thus $T(\alpha\vert h\vert^{p})\in X_{j}$. It follows from the
definition of $X_{j}^{(p)}$ that $\left(T(\alpha\vert h\vert^{p})\right)^{\frac{1}{p}}\in X_{j}^{(p)}$,
and so $L(h)\in X_{j}^{(p)}$ 
by
the lattice property.

Furthermore, 
\begin{eqnarray}
\Vert L(h)\Vert_{X_{j}^{(p)}} & \leq & \left\Vert H(h)\right\Vert _{X_{j}^{(p)}}\nonumber \\
 & = & \left\Vert \left(T(\alpha\vert h\vert^{p})\right)^{\frac{1}{p}}\right\Vert _{X_{j}^{(p)}}\nonumber \\
 & = & \left(\left\Vert \left|\left(T(\alpha\vert h\vert^{p})\right)^{\frac{1}{p}}\right|^{p}\right\Vert _{X_{j}}\right)^{\frac{1}{p}}\nonumber \\
 & = & \left(\left\Vert T(\alpha\vert h\vert^{p})\right\Vert _{X_{j}}\right)^{\frac{1}{p}}\label{eq:estimating the norm of L(h)}\\
 & \leq & \left(\Vert T\Vert_{X_{j}\rightarrow X_{j}}\cdot\left\Vert \alpha\vert h\vert^{p}\right\Vert _{X_{j}}\right)^{\frac{1}{p}}\nonumber \\
 & = & \alpha^{\frac{1}{p}}\left(\Vert T\Vert_{X_{j}\rightarrow X_{j}}\right)^{\frac{1}{p}}\cdot\left(\left\Vert \vert h\vert^{p}\right\Vert _{X_{j}}\right)^{\frac{1}{p}}\nonumber \\
 & = & \alpha^{\frac{1}{p}}\left(\Vert T\Vert_{X_{j}\rightarrow X_{j}}\right)^{\frac{1}{p}}\left\Vert h\right\Vert _{X_{j}^{(p)}}\nonumber 
\end{eqnarray}
which proves that $L:(X_{0}^{(p)},X_{1}^{(p)})\rightarrow(X_{0}^{(p)},X_{1}^{(p)})$
is bounded. In addition, if $\left(X_{0},X_{1}\right)$ is a positive
$C$-Calder\'on couple, the preceding estimates and (\ref{eq:temp})
show that $\left\Vert L\right\Vert _{X_{j}^{(p)}\to X_{j}^{(p)}}\le\alpha^{\frac{1}{p}}C^{\frac{1}{p}}$
and therefore that $(X_{0}^{(p)},X_{1}^{(p)})$ is a $\left(\alpha C\right)^{^{\frac{1}{p}}}$-Calder\'on
couple. According to (\ref{eq:Maligranda}), we can set $\alpha=2^{p-1}$.
This completes the proof of Theorem \ref{thm:a banach couple remains a banach couple}.
$\qed$

\begin{rem*} The interplay which has served us here, between sublinear
and linear operators enabled by an appropriate version of the Hahn-Banach
extension property, has also been used elsewhere in interpolation
theory, in particular to show that certain interpolation theorems
which are valid for linear operators also hold for sublinear operators.
For example the theorem of Janson on p.~52 of \cite{jasl} deals
with the case of sublinear operators mapping into couples of $L^{p}$
spaces. That theorem has been extended to the case of other couples
of Banach or quasi-Banach lattices by Bukhvalov \cite{Bukhvalov}
and Masty{\l{}}o \cite[Theorem 4.2, p. 416]{Mastylo}. \end{rem*}
\bigskip{}

\section{APPENDIX}

In this section we identify several conditions which ensure that a
given Banach lattice of measurable functions is also a complete lattice.
We are quite sure that several (and maybe even all) of the results
in this section are already known. However, since we have not found
references for them thus far, we add them and their proofs here for
completeness. We invite the reader to inform us of any relevant literature.
\begin{lem}
\label{lemma: A separable Banach lattice is a complete lattice}A
separable Banach lattice of measurable functions is a complete lattice. 
\end{lem}
\begin{proof} Suppose $X$ is a Banach lattice of measurable functions,
and that $Q\subseteq X$ is bounded from above by, say, $x\in X$.
Since $X$ is separable, there is a countable set $D$ such that $D\subseteq Q\subseteq\overline{D}$.
Let us write $D=\left\{ d_{n}\right\} _{n\in\mathbb{N}}$ and define
$y$ as the pointwise supremum of the elements of $D$. That is, for
every $\omega\in\Omega$ we define $y(\omega)=\mbox{sup}_{n\in\mathbb{N}}\left\{ d_{n}(\omega)\right\} $.

First we note that $y$ is an element of $X$: Indeed, $y$ is a measurable
function, as a pointwise supremum of a countable collection
of
measurable functions. In addition, since 
$d_1
\leq y\leq x$, the lattice property
guarantees that $y\in X$.

Secondly we show that $y$ is an upper bound of $Q$: If $q\in Q$
then there is a sequence $\left(e_{m}\right)_{m\in\mathbb{N}}$ of
elements of $D$ such that 
\begin{equation}
\mbox{lim}_{m\rightarrow\infty}\left\Vert q-e_{m}\right\Vert =0\,.\label{eq:Conv_in_norm}
\end{equation}
Since $X$ is a Banach lattice of measurable functions, (\ref{eq:Conv_in_norm})
implies that there exists a subsequence $\left(e_{m_{k}}\right)_{k\in\mathbb{N}}$
such that $e_{m_{k}}\underset{k\rightarrow\infty}{\rightarrow}q$
almost everywhere. (The proof of this standard fact can be seen, e.g.
as one of the steps in the proof of Theorem 2 in \cite{Zaanen}, Chapter
15, Section 64, p. 445.) In other words, there is a measurable set
$B\in\Sigma$ such that $\mu(B)=0$ and $\mbox{lim}_{k\rightarrow\infty}e_{m_{k}}(\omega)=q(\omega)$
for every $\omega\in\Omega\setminus B$. Clearly, $q\leq\mbox{sup}_{k\in\mathbb{N}}\left\{ e_{m_{k}}\right\} \leq\mbox{sup}_{n\in\mathbb{N}}\left\{ d_{n}\right\} =y$
(almost everywhere), as required.

Thirdly, we show that if $z\in X$ is an upper bound of $Q$ then
$y\leq z$: This is almost trivial, since if $z$ is an upper bound
of $Q$ then $d_{n}\leq z$ for every $n\in\mathbb{N}$, hence $y=\mbox{sup}_{n\in\mathbb{N}}\left\{ d_{n}\right\} \leq z$.
\end{proof}

\begin{claim} Let $X$ be a Banach lattice of real measurable functions
on a measure space $\left(\Omega,\Sigma,\mu\right)$. If $\left(\Omega,\Sigma,\mu\right)$
is $\sigma$-finite, then $X$ is a complete lattice.\end{claim}

\begin{proof} The easy proof of this claim follows from the fact
that any collection of measurable functions which is bounded from
above has a unique least upper bound (to within a set of measure zero)
when the underlying measure space is $\sigma$-finite (see \cite{Doob},
Chapter V, Section 18, pp.~71-72). \end{proof}

The following simple result will be helpful for dealing with complete
lattices.

\begin{claim} \label{Clm:NonNegative}A Banach lattice $X$ is a
complete lattice if and only if every subset of non-negative functions
in $X$ which has an upper bound also has a least upper bound.\end{claim}

\begin{proof} Let $A$ be a subset of $X$. Let $g_{0}$ be some
element in $A$. Then, obviously, an element $f\in X$ is an upper
bound of $A$ if and only if $f$ is an upper bound of $A_{0}:=\left\{ \max\left\{ g,g_{0}\right\} :g\in A\right\} $.
It can also be readily seen that the element $f$ is a least upper
bound of $A$ if and only if this same element is a least upper bound
of $A_{0}$. The set $B:=\left\{ g-g_{0}:g\in A_{0}\right\} =\left\{ \max\left\{ g,g_{0}\right\} -g_{0}:g\in A\right\} $
consists of non-negative elements of $X$ and has an upper bound if
and only if $A_{0}$ has an upper bound. Furthermore $f$ is a least
upper bound of $A_{0}$ if and only if $f-g_{0}$ is a least upper
bound of $B$. Claim \ref{Clm:NonNegative} is an obvious consequence
of the preceding observations. \end{proof}

For a given Banach lattice $X$ of measurable functions on some measure
space $\left(\Omega,\Sigma,\mu\right)$ let us use the notation 
$\Omega_{f}:=\left\{ \omega\in\Omega
\,\vert \,
f(\omega)\ne0\right\} $
for the support of an element $f\in X$. Then let $L^{\infty}(\Omega_{f})$
denote the subspace of $L^{\infty}\left(\Omega,\Sigma,\mu\right)$
consisting of all essentially bounded functions that vanish on $\Omega\setminus\Omega_{f}$.
Perhaps a somewhat easier way to characterize a complete lattice is
the following:

\begin{claim} \label{Claim: X has LUBP iff Linfty(Omega) has LUBP}Suppose
$X$ is a Banach lattice of measurable functions. $X$ has the LUBP
if and only if $L^{\infty}(\Omega_{f})$ has the LUBP for all $f\in X$.\end{claim}

\begin{proof} Suppose first that\textbf{ }$X$ has the LUBP. Given
an arbitrary element $f_{0}\in X$ we have to show that $L^{\infty}(\Omega_{f_{0}})$
has the LUBP. Obviously, since $\Omega_{f}=\Omega_{\left|f\right|}$
we may assume without loss of generality that $f_{0}$ is non-negative.
Let $A$ be an arbitrary subset of non-negative elements of $L^{\infty}(\Omega_{f_{0}})$
which is bounded above by some element $g_{0}\in L^{\infty}(\Omega_{f_{0}})$.
Let $B=\left\{ uf_{0}\vert u\in A\right\} $ (where $uf_{0}$ denotes
pointwise multiplication a.e. of the two functions $u$ and $f_{0}$).
By the lattice property of $X$ we see that $B$ is a subset of $X$.
Of course $B$ contains only non-negative elements and it is bounded
above by $g_{0}f_{0}$. Therefore there exists an element $h_{0}\in X$
which is a least upper bound of $B$. In particular 
\begin{equation}
h_{0}\le g_{0}f_{0}\,.\label{eq:dcw}
\end{equation}
 Let us now define $u_{0}:\Omega\to[0,\infty)$ by $u_{0}=\chi_{\Omega_{f_{0}}}\cdot\frac{h_{0}}{f_{0}}$.
In view of (\ref{eq:dcw}) we have that $u_{0}\le g_{0}$ and therefore
$u_{0}$ is essentially bounded on $\Omega$ and vanishes a.e. on
$\Omega\setminus\Omega_{f_{0}}$. It is easy to see that the function
$u_{0}$, or rather the equivalence class of which it is a representative,
is a least upper bound of $A$ in $L^{\infty}\left(\Omega_{f_{0}}\right)$.
Consequently $L^{\infty}\left(\Omega_{f_{0}}\right)$ has the LUBP.

Now suppose, conversely, that $L^{\infty}\left(\Omega_{f}\right)$
has the LUBP for each element $f\in X$. Let $A$ be an arbitrary
subset of non-negative elements of $X$ which is bounded from above
by some element $f_{0}\in X$. Of course $f_{0}$ must be non-negative
and every element of $A$ must vanish a.e on $\Omega\setminus\Omega_{f_{0}}$.
Let $B:=\left\{ \chi_{\Omega_{f_{0}}}\cdot\frac{u}{f_{0}}
\,
\vert
\, 
u\in A\right\} $.
Obviously each $g\in B$ is essentially bounded and in fact satisfies
$0\le g\le\chi_{\Omega_{f_{0}}}$ a.e. Therefore, since $L^{\infty}\left(\Omega_{f_{0}}\right)$
is a complete lattice, we deduce that $B$ has a least upper bound
$g_{0}\in L^{\infty}\left(\Omega_{f_{0}}\right)$. It is a trivial
matter to check that the measurable function $h_{0}:=g_{0}f_{0}$
is a least upper bound in $X$ of $A$ and this completes the proof.
\end{proof}

Combining Remark \ref{rmk: examples of complete lattices. } and Claim
\ref{Claim: X has LUBP iff Linfty(Omega) has LUBP} one may easily
deduce the following:
\begin{cor}
Given a Banach lattice of measurable functions $X$, if the support
of every element in $X$ is $\sigma$-finite, then $X$ is a complete
lattice. 
\end{cor}
\begin{rem} \label{Proof that X is a complete lattice iff X^(p) is}Once
we have proven Claim \ref{Claim: X has LUBP iff Linfty(Omega) has LUBP}
the proof of Claim \ref{Clm: X is complete iff X^p is-1} immediately
follows from the fact that $\Omega_{f}=\Omega_{\left|f\right|}=\Omega_{\left|f\right|^{\frac{1}{p}}}$
for each measurable $f$. \end{rem}

We may now apply Claim \ref{Clm:NonNegative} to prove Claim \ref{Clm: X_0+X_1 is complete iff X_0, X_1 are-1}:

\begin{proof} We first assume that $X_{0}$ and $X_{1}$ are two
complete lattices, and prove that $X_{0}+X_{1}$ is a complete lattice.

In view of Claim \ref{Clm:NonNegative} it will suffice to show that
if $A\subseteq X_{0}+X_{1}$ is any collection of non-negative functions
which is bounded from above, then $A$ has a least upper bound in
$X_{0}+X_{1}$.

Let $f\in X_{0}+X_{1}$ be an upper bound of $A$. It is well known
(cf. (\ref{equi. of K and D})) that there is a measurable set $E$
such that $f\chi_{E}\in X_{0}$ and $f\chi_{\Omega\setminus E}\in X_{1}$.
We now define 
\begin{eqnarray*}
A_{0} & = & \left\{ a\chi_{E}
\,
\vert
\, 
a\in A\right\} \\
A_{1} & = & \left\{ a\chi_{\Omega\setminus E}
\,
\vert 
\,
a\in A\right\} \,.
\end{eqnarray*}

Clearly $0\leq a\leq f$ for every $a\in A$, hence $0\leq a\chi_{E}\leq f\chi_{E}$
and $0\le a\chi_{\Omega\setminus E}\leq f\chi_{\Omega\setminus E}$,
hence $A_{0}$ is a bounded subset of $X_{0}$ and $A_{1}$ is a bounded
subset of $X_{1}$. Since $X_{0},\, X_{1}$ are complete lattices,
$A_{0}$ and $A_{1}$ both have least upper bounds, respectively $b_{0}\in X_{0}$
and $b_{1}\in X_{1}$.

It is easy to verify that $b=b_{0}+b_{1}\in X_{0}+X_{1}$ is a least
upper bound of $A$, and therefore that $X_{0}+X_{1}$ is a complete
lattice.

We now turn to the second part of the proof. We assume that $X_{0}+X_{1}$
is a complete lattice, and prove that both $X_{0}$ and $X_{1}$ are
complete lattices.

Indeed, suppose that $A\subseteq X_{0}$ is bounded from above by
an element of $X_{0}$. That is, there 
exists an element
$g\in X_{0}$ such that
$f\leq g$ for every $f\in A$. Since clearly $g\in X_{0}+X_{1}$,
the set $A$ is bounded from above as a subset of $X_{0}+X_{1}$ and
hence it has a least upper bound, say $h\in X_{0}+X_{1}$. Since for
any $f\in A$ we have $f\leq h\leq g$, the lattice property of $X_{0}$
implies that $h\in X_{0}$. It is clear that the element $h$ is a
least upper bound of $A$ with respect to $X_{0}$. Thus $X_{0}$
is a complete lattice, and an analogous proof shows the same for $X_{1}$.

This completes the proof. \end{proof}

\textit{Acknowledgements:} We thank Alexander Ioffe and Mieczys{\l{}}aw
Masty{\l{}}o for some helpful discussions. We also thank Lech Maligranda
for some helpful remarks, as mentioned in Section \ref{sec:main}.

\end{document}